\documentclass[a4paper,10pt]{article}
\usepackage[english]{babel}
\usepackage[latin1]{inputenc}
\usepackage[T1]{fontenc}
\usepackage{amsmath}
\usepackage{comment}
\usepackage{amsfonts}
\usepackage{amsthm}
\usepackage{mathrsfs}
\usepackage{amssymb}
\usepackage{graphicx}
\usepackage{tikz-cd}
\usetikzlibrary{calc}
\usepackage{marginnote}

\theoremstyle{thmstyleone}%
\newtheorem{theorem}{Theorem}
\theoremstyle{thmstyletwo}%
\newtheorem{example}{Example}%
\newtheorem{remark}{Remark}%
\newtheorem{lemma}{Lemma}%

\theoremstyle{thmstylethree}%
\newtheorem{definition}{Definition}%

\renewcommand{\epsilon}{\varepsilon}

\title{Rademacher's Theorem for Calderon-Zygmund-type Spaces}

\author{T. Lamby\footnote{University of Luxembourg, Department of Mathematics -- Maison du Nombre, 6, Avenue de la Fonte, 4364 Esch-sur-Alzette, Luxembourg. thomas.lamby@uni.lu}}

\begin{document}

\maketitle

\begin{abstract}
Rademacher's Theorem can be interpreted as an almost-everywhere \emph{little-$o$ improvement principle}: if a function admits a uniform pointwise first-order Lipschitz control at every point, then this control improves to a vanishing one at almost every point. In the language of Calder\'on--Zygmund pointwise spaces, this means that
\[
f \in T^\infty_1(x) \quad \forall x \in \mathbb{R}^d
\qquad \Longrightarrow \qquad
f \in t^\infty_1(x) \quad \text{for a.e. } x \in \mathbb{R}^d.
\]

The purpose of this paper is to establish an analogous almost-everywhere improvement principle in a refined $L^p$ setting. We consider pointwise Calder\'on-Zygmund spaces $T^p_{\phi}(x)$ defined via polynomial approximation in $L^p$ with a function parameter $\phi$, allowing for fractional regularity indices and logarithmic corrections through Boyd functions. We prove that, under natural assumptions on $\phi$, the uniform membership
\[
f \in T^p_{\phi}(x) \quad \forall x \in E
\]
on a measurable set $E \subset \mathbb{R}^d$ implies an almost-everywhere improvement to a vanishing approximation rate, namely
\[
f \in t^p_{\phi,n+1}(x) \quad \text{for a.e. } x \in E,
\]
where $n < \underline{b}(\phi) \leq \overline{b}(\phi) < n+1$.

The proof combines measurability arguments, a generalized Whitney extension theorem, and fine properties of Sobolev spaces. We also show that this result is essentially sharp: in general, one cannot expect almost-everywhere membership in $t^p_{\phi,n}(x)$ for fractional indices, and explicit counterexamples are provided.
\end{abstract}


\section{Introduction}\label{sec1}

Pointwise Calder\'on--Zygmund spaces provide a flexible framework to quantify local regularity through polynomial approximation in $L^p$. They were originally introduced by Calder\'on and Zygmund \cite{Calderon:61}. Given a point $x_0 \in \mathbb{R}^d$, $p \in [1,\infty]$ and a real number $u \geq -d/p$, the space $T^p_u(x_0)$ consists of functions $f \in L^p(\mathbb{R}^d)$ for which there exists a polynomial $P$ of degree strictly less than $u$ such that
\begin{equation}\label{equa Tpu}
r^{-d/p}\|f-P\|_{L^p(B(x_0,r))}\leq C\,r^u \quad \forall r>0.
\end{equation}
If, in addition,
\begin{equation}\label{equa tpu}
r^{-d/p}\|f-P\|_{L^p(B(x_0,r))}=o(r^u) \quad \text{when } r\rightarrow 0^+,
\end{equation}
then $f$ is said to belong to the space $t^p_u(x_0)$.

These spaces allow one to describe the pointwise regularity of functions that may fail to be locally bounded (see for instance \cite{Jaffard:18,Seuret:17,Lamby:21}). Since the spaces $T^p_u(x_0)$ are decreasing with respect to $u$, one can define the \emph{$p$-exponent} of a function $f$ at $x_0$ by
\[
h_p^f(x_0)=\sup\{u\geq -d/p : f\in T^p_u(x_0)\}.
\]
When $p=\infty$, the space $T^\infty_u(x_0)$ coincides with the classical pointwise H\"older space $\Lambda^u(x_0)$.

\medskip

From this point of view, Rademacher's Theorem can be interpreted as an \emph{almost-everywhere improvement principle}. It states that if a function is uniformly Lipschitz, then its first-order approximation improves from a bounded rate to a vanishing one at almost every point. In terms of Calder\'on--Zygmund spaces, this means that
\[
f \in T^\infty_1(x) \quad \forall x \in \mathbb{R}^d
\qquad \Longrightarrow \qquad
f \in t^\infty_1(x) \quad \text{for a.e. } x \in \mathbb{R}^d,
\]
where the polynomial involved in the definition of $t^\infty_1(x)$ has degree at most~$1$.

This principle has been extended in several directions. In particular, Calder\'on showed that if $\Omega \subset \mathbb{R}^d$ is bounded and open, then every function in the Sobolev space $W^p_1(\Omega)$ is differentiable almost everywhere whenever $p>d$ \cite{Evans:15}. This result follows from the Lebesgue differentiation theorem together with Sobolev embeddings, and Rademacher's Theorem appears as a limiting case, since Lipschitz functions belong to $W^\infty_1(\Omega)$. Rademacher-type results also exist in more general metric settings, where differentiability is formulated in terms of metric differentials \cite{Cheeger:99}.

\medskip

The purpose of this paper is to establish an analogous almost-everywhere improvement principle in a refined $L^p$ framework. More precisely, we replace the Lipschitz assumption by a pointwise control in generalized Calder\'on--Zygmund spaces $T^p_\phi(x)$, where the power function $r^u$ in \eqref{equa Tpu} is replaced by a function parameter $\phi$. These spaces allow for fractional regularity indices and finer scales involving logarithmic corrections. They are defined using Boyd functions and will be introduced in detail in Section~\ref{sec2}.

Given a measurable set $E \subset \mathbb{R}^d$, the guiding question of this work is the following:
\begin{quote}
\emph{If a function $f$ satisfies $f \in T^p_\phi(x)$ for every $x \in E$, does this uniform pointwise control improve to a vanishing one at almost every point of $E$?}
\end{quote}
For integer indices, this question was answered positively by Ziemer in \cite{Ziemer:89} for the spaces $T^p_u(x)$ with $u \in \mathbb{N}_0$, yielding an $L^p$ version of Rademacher's Theorem.

\medskip

Our strategy follows the classical approach of Ziemer, adapted to the function-parameter setting. First, we show that the map
\[
x \longmapsto \|f\|_{T^p_\phi(x)}
\]
is measurable. Using Lusin's Theorem, this allows us to localize the problem on compact subsets where uniform bounds are available. We then rely on a generalized Whitney extension theorem to construct a smooth extension agreeing almost everywhere with $f$ on these subsets. Fine properties of Sobolev functions are subsequently used to obtain an almost-everywhere vanishing approximation rate for the extension, which is finally transferred back to the original function. 

For fractional indices, this approach yields an essentially optimal result. In general, one can only expect almost-everywhere membership in the space $t^p_{\phi,n+1}(x)$, where
\[
n < \underline{b}(\phi) \leq \overline{b}(\phi) < n+1,
\]
and where the polynomial appearing in the definition of the space has degree at most $n+1$. We further provide explicit counterexamples showing that this threshold cannot, in general, be lowered.

\medskip

Throughout the paper, $\mathcal{L}$ denotes the Lebesgue measure, $\mathbb{N}_0=\{0,1,2,3,\dots\}$ and $\mathbb{N}=\{1,2,3,\dots\}$.

\section{Calder\'on-Zygmund spaces with Function parameter}\label{sec2}

A crucial aspect in order to refine regularity spaces is the notion of Boyd functions \cite{Boyd:67,Boyd:69,Johnson:79,Krein:82,Merucci:84, Lamby:22, Lamby:24}, which serve as a replacement for the traditional functions $t \mapsto t^\theta$ typically used in these settings: a function $\phi : (0, \infty) \to (0, \infty)$ is a Boyd function if it is continuous, $\phi(1) = 1$ and 
\[
\overline{\phi} (t) := \sup_{s>0} \frac{\phi(st)}{\phi(s)} < \infty, \quad \forall t \in (0, \infty).
\]
The lower and upper Boyd indices of a Boyd function $\phi$ are then defined by
\[
 \underline{b}(\phi):=\sup_{t<1} \frac{\log \bar \phi(t)}{\log t} = \lim_{t\to 0} \frac{\log \bar \phi(t)}{\log t}
\]
and
\[
 \overline{b}(\phi):=\inf_{t>1} \frac{\log \bar \phi(t)}{\log t} = \lim_{t\to \infty} \frac{\log \bar \phi(t)}{\log t},
\]
respectively.

\begin{definition}\label{def T^p_phi(x_0)}
Let $x_0\in\mathbb{R}^d$, $p\in[1,\infty]$ and $\phi\in\mathcal{B}$ such that $\underline{b}(\phi)>-\frac{d}{p}$. A function $f\in L^p(\mathbb{R}^d)$ belongs to the space $T^p_\phi(x_0)$ if there exists a polynomial $P$ with degree strictly smaller than $\underline{b}(\phi)$ and a constant $C>0$ such that
\begin{equation}\label{equation de base espace tpphi}
r^{-d/p}\|f-P\|_{L^p(B(x_0,r))}\leq C\,\phi(r)\quad\forall r>0.
\end{equation}
\end{definition}
For example, when $\underline{b}(\phi)$ is non-integer, the polynomial involved is of degree at most $\lfloor \underline{b}(\phi)\rfloor$.
The condition $\underline{b}(\phi)>-\frac{d}{p}$ is here to ensure that the spaces are not degenerate. One can easily check that the polynomial appearing in (\ref{equation de base espace tpphi}) is unique. We set 
\[\Lambda^\phi(x_0)=T^\infty_\phi(x_0).\]
Let $f\in T^p_\phi(x_0)$ and $P$ be the polynomial relating to $f$, then $P$ can be written as
\[P(x)=\sum_{|\alpha|<\underline{b}(\phi)}\frac{D^\alpha P(x_0)}{|\alpha| !}(x-x_0)^\alpha.\]
We set
\[|f|_{T^p_\phi(x_0)}=\sup_{r>0}\phi(r)^{-1}r^{-d/p}\|f-P\|_{L^p(B(x_0,r))},\]
and
\[\|f\|_{T^p_\phi(x_0)}=\|f\|_{L^p(\mathbb{R}^d)}+\sum_{|\alpha|<\underline{b}(\phi)}\frac{|D^\alpha P(x_0)|}{|\alpha| !}+|f|_{T^p_\phi(x_0)}.\]
Then it is straightforward that $(T^p_\phi(x_0),\|\cdot\|_{T^p_\phi(x_0)})$ is a Banach space. Let $t^p_{\phi,n}(x_0)$, consisting of functions $f\in L^p(\mathbb{R}^d)$ such that there exists a polynomial $P$ of degree at most $n$ such that
\[r^{-d/p}\|f-P\|_{L^p(B(x_0,r))}=o(\phi(r))\text{\emph{ when }} r\rightarrow 0^+.\]
If $n<\underline{b}(\phi)$, it is easy to prove that $t^p_{\phi,n}(x_0)$ is a closed subspace of $T^p_{\phi}(x_0)$ with respect to the topology induced by the norm $\|\cdot\|_{T^p_{\phi}(x_0)}$ (see \cite{Loosveldt:20}). 

\begin{remark}
If $\phi(t)=t^u$, with $u \in \mathbb{N}_0$, then $t^p_{\phi,u}(x_0)$ is the usual space $t^p_u(x_0)$ of \cite{Ziemer:89}. 
Note that, in this case, the space $T^p_\phi(x_0)$ involves a polynomial of degree at most $u-1$ whereas the space $t^p_u(x_0)$ involves a polynomial of degree at most $u$, so that $t^p_u(x_0)$ is not a true subspace of $T^p_\phi(x_0)$.
\end{remark}

 These spaces are used to capture finer notions of regularity. For example (see \cite{Morters:10}), if $B$ is a Brownian motion on a probability space $(\Omega, \mathcal{F}, \mathbb{P})$, then almost surely and for almost every $t_0 \in \mathbb{R}$, there exists a constant $C > 0$ such that
\begin{equation}\label{equa Brownien}
|B(t) - B(t_0)| \leq C \big(|t - t_0| \log \log 1/|t - t_0|\big)^{1/2},
\end{equation}
whereas for every $t_0 \in \mathbb{R}$, the H\"older exponent of $B$ at $t_0$ equals $h(t_0) = \frac{1}{2}$. Therefore, if $\phi\in\mathcal{B}$ is such that $\phi(t)\sim (t \log \log 1/|t|)^{1/2}$ for all $t>0$, then $T^\infty_\phi$ is the adapted functional space that take into account the iterated logarithm in the pointwise regularity of $B$, a phenomenon that classical H\"older spaces fail to capture.

 \medskip
\noindent
The following result, proved in \cite{Loosveldt:20}, is a function-parameter 
version of Whitney's extension theorem. Roughly speaking, it shows that 
uniform $T^p_\phi$ conditions on a closed set $E$ ensure that the function 
coincides almost everywhere on $E$ with a $C^n$ function defined on a 
neighbourhood of $E$, together with quantitative control on finite differences 
of the derivatives of this $C^n$ extension. They have direct 
applications in the fine regularity theory with function parameters, in 
particular for obtaining pointwise regularity estimates for solutions of 
elliptic partial differential equations in the framework of $T^p_\phi$ spaces.

\begin{theorem}\label{Gen Whitney}\emph{(Generalization of Whitney's Extension Theorem)}
\\Let $E\subseteq \mathbb{R}^d$ be a closed set, $U=\{x\in\mathbb{R}^d\mid d(x,E)<1\}$, $p\in [1,\infty]$, $n\in\mathbb{N}_0$ and $\phi\in\mathcal{B}$ such that $n<\underline{b}(\phi)$.
\\If $f\in T^p_\phi(x_0)$ verifies $\|f\|_{T^p_\phi(x_0)}\leq M$ for $M>0$ and for all $x_0\in E$, then there exists $F\in C^n(U)$ such that $F=f$ almost everywhere on $E$. Moreover, if $m \in \mathbb{N}$ is such that $n < \underline{b}(\phi) \leq \overline{b}(\phi) < m$, then there exists a constant $C > 0$ such that for all $x \in U$ and $h \in \mathbb{R}^d\backslash\{0\}$ satisfying $[x, x + (m-n)h]\subseteq U$, we have
\[
|\Delta^{\,m-n}_h D^\alpha F(x)| \leq C \, \phi(|h|) \, |h|^{-n}
\]
for every multi-index $\alpha$ with $|\alpha| = n$.

\end{theorem}
 
We generalize here to the spaces $T^p_\phi(x_0)$ the approach carried out for spaces $T^p_u(x_0)$ and $t^p_u(x_0)$, $u \in \mathbb{N}_0$, in \cite{Ziemer:89}. This generalization proceeds in two directions: the indices are allowed to be non-integer, and one may also consider more refined regularity spaces, with the flexibility that a weight can exhibit more nuanced behavior than a pure power function; for instance, involving iterated logarithmic terms.

\section{Useful Results}\label{sec3}

The following Lemma is given in \cite{Ziemer:89}.
\begin{lemma}\label{lemme avec fonction test conv poly est le poly}
Let $n \in \mathbb{N}_0$. There exists a function $\varphi \in \mathcal{D}(\mathbb{R}^d)$ with support contained in $\overline{B(0,1)}$ such that for every polynomial $P$ of degree less than or equal to $n$ and for every $\varepsilon > 0$, we have $\varphi_\varepsilon \ast P = P$, where $\varphi_\varepsilon = \varepsilon^{-d} \varphi\left(\frac{\cdot}{\varepsilon}\right)$.
\end{lemma}

\begin{lemma}\label{lemme 1 Rademacher}
Let $E \subseteq \mathbb{R}^d$ be a measurable set, $p \in [1,\infty]$, and $\phi \in \mathcal{B}$ such that $\underline{b}(\phi) > -d/p$. If $f \in T^p_\phi(x_0)$, then the function
\[
\xi_f : E \rightarrow [0,\infty) : x \mapsto \|f\|_{T^p_\phi(x)}
\]
is measurable.
\end{lemma}
\begin{proof}
Let $x \in E$. We denote by $P_x$ the polynomial associated with $f$ at $x$, of degree strictly less than $\underline{b}(\phi)$.
It therefore suffices to show that for every multi-index $\alpha$ with $|\alpha| < \underline{b}(\phi)$,
\[
x \in E \mapsto D^\alpha P_x(x)
\]
is measurable.  
If $\underline{b}(\phi) \leq 0$, the claim is obvious since $P_x$ is the zero polynomial for every $x \in E$. We may thus assume that $\underline{b}(\phi) > 0$.
\smallskip
\\Let $\varphi$ be the function from Lemma \ref{lemme avec fonction test conv poly est le poly} for $n=\lfloor\underline{b}(\phi)\rfloor$. \smallskip
We define, for every $x \in E$,
\[
R_x = f - P_x.
\]
For $\varepsilon > 0$ and fixed $x \in E$, we have
\begin{align*}
D^\alpha \left(\varphi_\varepsilon \ast f(x)\right)
&= D^\alpha \left(\varphi_\varepsilon \ast P_x(x)\right) + D^\alpha \left(\varphi_\varepsilon \ast R_x(x)\right) \\
&= D^\alpha P_x(x) + \int_{\mathbb{R}^d} \varepsilon^{-d - |\alpha|} (D^\alpha \varphi)\left(\frac{y}{\varepsilon}\right) R_x(x - y) \, dy.
\end{align*}
But, since $|\alpha|<\underline{b}(\phi)$, using properties of Boyd functions, we get 
\begin{align*}
\left|\int_{\mathbb{R}^d}\varepsilon^{-d-|\alpha|}\,(D^\alpha \varphi)\left(\frac{y}{\varepsilon}\right)\,R_x(x-y)\,dy\right|
&\leq C_\varphi\,\varepsilon^{-d-|\alpha|}\,\int_{B(0,\varepsilon)}|R_x(x-y)|\,dy\\
&=C_\varphi\,\varepsilon^{-|\alpha|}\,\varepsilon^{-d}\,\|f-P_x\|_{L^1(B(x,\varepsilon))}\\
&\leq C_\varphi\,\varepsilon^{-|\alpha|}\,C_d^{1-1/p} \,\varepsilon^{-d/p}\,\|f-P_x\|_{L^p(B(x,\varepsilon))}\\
&\leq C'\,\varepsilon^{-|\alpha|}\,\phi(\varepsilon)\\
&\leq C''\,\varepsilon^{-|\alpha|}\,\varepsilon^{\underline{b}(\phi)-((\underline{b}(\phi)-|\alpha|)/2)}\\
&= C'' \varepsilon^{(\underline{b}(\phi)-|\alpha|)/2)}\underset{\varepsilon\rightarrow 0^+}{\longrightarrow}0.
\end{align*}
Therefore, for all $x\in E$, $D^\alpha\left(\varphi_\varepsilon \ast f(x)\right)$ tends to $D^\alpha P_x(x)$, hence the conclusion.
\end{proof}

\medskip
\noindent
The following lemma is technically involved but plays a key role in the sequel. 
Its proof relies on a localization argument combined with covering techniques. 
Starting from the uniform control of the local $L^p$ averages of $f$ by $\phi(r)$, 
one reduces to a bounded setting where $f$ has compact support and identifies 
a large subset on which $f$ vanishes almost everywhere. Covering arguments 
then yield suitable integrability estimates for a singular kernel involving 
$\phi$, leading to the desired $o(\phi(r))$ decay for almost every $x\in E$ 
as $r\to0^+$.

\begin{lemma}\label{lemme 2 Rademacher}
Let $p\in [1,\infty)$, $\phi\in\mathcal{B}$ such that $\underline{b}(\phi)>0$ and $f\in L^p(\mathbb{R}^d)$, if 
\[r^{-d/p}\,\|f\|_{L^p(B(x,r))}\leq C\,\phi(r)\quad\forall r>0\]
for all $x$ in a measurable set $E$ of $\mathbb{R}^d$, then for almost all $x\in E$,
\[r^{-d/p}\|f\|_{L^p(B(x,r))}= o(\phi(r))\text{\emph{ when }} r\rightarrow 0^+.\]
\end{lemma}
\begin{proof}
We may suppose that $E$ is bounded and that $f$ is with compact support. Since $\underline{b}(\phi)>0$, there exists a constant $C_\star>0$ such that 
\[\overline{\phi}(r)\leq C_\star\, r^{\underline{b}(\phi)-\underline{b}(\phi)/2}\leq C_\star\quad\forall r\in(0,1].\]
Let $\varepsilon >0$, we need to prove that for almost all $x\in E$, 
\begin{equation}\label{equa 1 lemme 2 rad}
r^{-d/p}\|f\|_{L^p(B(x,r))}<\varepsilon\,\phi(r)
\end{equation}
for sufficiently small $r$.
\\Let $A$ be a closed subset of $E$ such that
\[\mathcal{L}\,(E\backslash A)<\epsilon.\]
It is sufficient to prove the result for almost all $x\in A$. For $x\in A$, one has
\[\lim_{r\rightarrow 0^+}r^{-d}\,\|f\|_{L^1(B(x,r))}\leq C_d^{1-1/p}\lim_{r\rightarrow 0^+}r^{-d/p}\,\|f\|_{L^p(B(x,r))}\leq C'\,\lim_{r\rightarrow 0^+}\phi(r)=0.\]
Therefore, since almost all points are Lebesgue points, $f=0$ almost everywhere on $A$.
\smallskip\\We define the open set $U$ by
\[U=\{x\in\mathbb{R}^d:d(x,A)<1\}\]
and for all $x\in U\backslash A$, we set
\[h(x)=\frac{1}{10}\,d(x,A).\]
Using classical covering Theorems (see for example Theorem 1.3.1 of \cite{Ziemer:89}), there exists a countable set $S\subseteq U\backslash A$ such that the family
\[\left\{\overline{B(s,h(s))}:s\in S\right\}\]
is disjoint and 
\[\bigcup_{s\in S} \overline{B(s,5h(s))}\supseteq U\backslash A.\]
We get
\begin{align*}
\int_A\int_U\frac{|f(y)|}{|x-y|^{d}\,\phi(|x-y|)}dy\,dx&\leq \int_A\int_{U\backslash A}\frac{|f(y)|}{|x-y|^{d}\,\phi(|x-y|)}dy\,dx\\
&\leq \int_A\sum_{s\in S}\int_{B(s,5h(s))}\frac{|f(y)|}{|x-y|^{d}\,\phi(|x-y|)}dy\,dx\\
&= \sum_{s\in S}\int_{B(s,5h(s))}|f(y)|\int_A\frac{dx}{|x-y|^{d}\,\phi(|x-y|)}dy.
\end{align*}
Since $A$ is closed, for all $s\in S$,  there exists $x_s\in A$ such that
\[|s-x_s|=d(s,A)=10h(s),\]
and
\[B(s,5h(s))\subseteq B(x_s,|s-x_s|+5h(s))=B(x_s,15 h(s)).\]
Consequently,
\begin{align*}
(5h(s))^{-d}\,\|f\|_{L^1(B(s,5h(s)))}&\leq C_d^{1-1/p}\,(5h(s))^{-d/p}\,\|f\|_{L^p(B(s,5h(s)))}\\&\leq C_d^{1-1/p}\,(5h(s))^{-d/p}\,\|f\|_{L^p(B(x_s,15h(s)))}\\
&\leq C' \phi(15h(s))\\
&\leq C''\phi(5h(s)).
\end{align*}
Remark that if $x\in A$ and $y\in B(s,5h(s))$, then
\[|x-y|\geq  d(s,A)-5h(s)=5h(s).\]
Thus, for $y\in B(s,5h(s))$, we get 
\begin{align*}
\int_A \frac{dx}{|x-y|^{d}\,\phi(|x-y|)}
&\leq C_1\int_{|5h(s)|}^{+\infty}\frac{r^{-1}}{\phi(r)}dr\\
&\leq C_1\,C_\star \,(5h(s))^{\underline{b}(\phi)/2}\int_{|5h(s)|}^{+\infty} \frac{r^{-1}}{r^{\underline{b}(\phi)/2}\,\phi(5h(s))}\,dr\\
&\leq C_2\,\phi(5h(s))^{-1}.
\end{align*}
We obtain
\begin{align*}
\int_A \int_{B(s,5h(s))}\frac{|f(y)|}{|x-y|^{d}\,\phi(|x-y|)}dy\,dx &\leq C_2\,(\phi(5h(s))^{-1}\int_{B(s,5h(s))}|f(y)|\,dy\\
&\leq C_2\,(\phi(5h(s))^{-1}\,(5(h(s))^d\,C''\,\phi(5h(s))\\
&\leq C_3\,h(s)^{d}.
\end{align*}
Since family $\left\{\overline{B(s,h(s))}:s\in S\right\}$ is disjointed,
\[\int_A\int_U\frac{|f(y)|}{|x-y|^{d}\,\phi(|x-y|)}dy\,dx\leq C_3\sum_{s\in S}h(s)^{d}\leq C_3\sum_{s\in S}h(s)<\infty,\]
so that
\[\int_U\frac{|f(y)|}{|x-y|^{d}\,\phi(|x-y|)}dy<\infty\]
for almost all $x\in A$. Since $f$ has compact support,
\[\int_{\mathbb{R}^d\backslash U}\frac{|f(y)|}{|x-y|^{d}\,\phi(|x-y|)}dy<\infty\]
for all $x\in A$ and thus
\[\int_{\mathbb{R}^d}\frac{|f(y)|}{|x-y|^{d}\,\phi(|x-y|)}dy<\infty\]
for almost all $x\in A$.
\\We observe that, to establish the above result, we have merely used the fact that
\[r^{-d}\,\|f\|_{L^1(B(x,r))}\leq C_d^{1-1/p}\,r^{-d/p}\,\|f\|_{L^p(B(x,r))}\leq C\phi(r).\]
Therefore, if $g=|f|^p$, the assumption becomes
\[r^{-d}\,\|g\|_{L^1(B(x,r))}\leq C\phi(r)^p\]
for all $x\in E$. 
\\But, if $\phi\in \mathcal{B}$, then $\phi^p\in  \mathcal{B}$ et $\underline{b}(\phi^p)=p\,\underline{b}(\phi)>0$.
\\Thus, by applying once again the argument used above, we obtain that
\[\int_{\mathbb{R}^d}\frac{|f(y)|^p}{|x-y|^{d}\,\phi^p(|x-y|)}\,dy=\int_{\mathbb{R}^d}\frac{|g(y)|}{|x-y|^{d}\,\phi^p(|x-y|)}\,dy<+\infty\]
for almost all $x\in E$.
\\For such $x$ and for all sufficiently small $r$, we have
\[\int_{B(x,r)}\frac{|f(y)|^p}{|x-y|^{d}\,\phi^p(|x-y|)}<\varepsilon^p\,C_\star^{-p},\]
and therefore, for all sufficiently small $r$,
\begin{align*}
r^{-d/p}\|f\|_{L^p(B(x,r))}&=\phi(r)\left(\int_{B(x,r)}\frac{|f(y)|^p}{r^{d}\,\phi^p(r)}\,dy\right)^{1/p}\\
&\leq \phi(r)\left(\int_{B(x,r)}\frac{|f(y)|^p}{r^{d}\,\phi^p(|x-y|)}\,dy\,C_\star^p\right)^{1/p}\\
& < \varepsilon\,\phi(r).
\end{align*}
\end{proof}
To generalize Rademacher's theorem, we will also make use of Lusin's Theorem, which establishes a connection between measurability and continuity for functions defined on locally compact Hausdorff spaces equipped with a Radon measure (\cite{Luzin:12}). We state a weaker version of this result for spaces $E$ contained in a Euclidean space $\mathbb{R}^d$, under the additional assumption that $E$ is bounded.
\begin{theorem}\emph{(Lusin's Theorem)}
\\A real-valued function $f$ defined on a bounded measurable subset $E$ of $\mathbb{R}^d$ is measurable if and only if, for every $\epsilon>0$, there exists a closed set $E_\epsilon\subseteq E$ such that
\[\mathcal{L}(E\backslash E_\epsilon)\leq \epsilon\] and such that $f$ is continuous on $E_\epsilon$.
\end{theorem}
We also need some preliminary results concerning Sobolev spaces. Let $\Omega$ be an open set of $\mathbb{R}^d$.
\begin{theorem}\label{Thm 1 ziemer sobolev}
If $f\in L^p(\Omega)$, then $f\in W^p_1(\Omega)$ if and only if there exists a function $F$ that coincides with $f$ almost everywhere on $\Omega$, which is absolutely continuous on almost every line segment of $\Omega$ axis-parallel and whose first-order partial derivatives belong to $L^p(\Omega)$.
\end{theorem}
The following result is also proved in \cite{Ziemer:89} (see Theorems 3.4.1 and 3.4.2 of \cite{Ziemer:89}) in a slightly more general setting where the null measure set in question is actually of zero Bessel capacity.
\begin{theorem}\label{Thm 2 ziemer sobolev}
If $p>1$ and $1\leq m_1\leq m_2$ verify $(m_2-m_1)\,p<d$ and $f\in W^p_{m_2}(\mathbb{R}^d)$, then there exists a polynomial $Q$ of degree at most $m_1$ such that
\[r^{-m_1}\,r^{-d/p}\,\left\|f-Q\right\|_{L^p(B(x,r))}\underset{r\rightarrow 0^+}{\longrightarrow}0\]
for almost all $x\in\mathbb{R}^d$.
\end{theorem}

We are now in a position to generalize Rademacher's Theorem. We follow the proof in \cite{Ziemer:89} since we will need similar arguments for the new results.

\begin{theorem}\label{Rademacher's Theorem Ziemer}\emph{(Rademacher's Theorem in the $L^p$ context)}
\\Let $E\subseteq \mathbb{R}^d$ measurable, $p\in(1,\infty)$ and $n\in\mathbb{N}_0$. If $f\in T^p_n(x)$ for all $x\in E$, then $f\in t^p_n(x)$ for almost all $x\in E$.
\end{theorem}
\begin{proof}~
\smallskip\\$\bullet$ Let us first assume that $E$ is bounded. By lemma \ref{lemme 1 Rademacher} and Lusin's Theorem, there exist compacts $E_{1/m}\subseteq E$, $m\in\mathbb{N}$, such that
\[\mathcal{L}\left(E\backslash E_{1/m}\right)\leq\frac{1}{m}\]
and such that \[\xi_f : E_{1/m} \rightarrow [0,\infty) : x\mapsto \|f\|_{T^p_\phi(x)}\] is continuous. Let $m\in\mathbb{N}$, there exists a constant $M=M(m)>0$ such that
 \[\|f\|_{T^p_n(x)}\leq M\quad\forall x\in E_{1/m}.\]
Let $x\in E_{1/m}$, we denote by $P_x$ the polynomial of degree strictly less than or equal to $n-1$ associated with $f$. Thanks to Whitney's extension Theorem \ref{Gen Whitney}, there exists an open set $U$ containing $E_{1/m}$ and a function $F\in C^{n-1}(U)$, with Lipchitz-continuous partial derivatives such that $F=f$ almost everywhere $E_{1/m}$ and such that for all $|\beta|< n$ and for almost all $x\in E_{1/m}$,
\[D^\beta F(x)=D^\beta P_x(x).\]
By Theorem \ref{Thm 1 ziemer sobolev}, $F\in W^p_{n,\text{loc}}(\mathbb{R}^d)$. Using Theorem \ref{Thm 2 ziemer sobolev} with $m_1=m_2=n$, it follows that for almost every $x\in \mathbb{R}^d$, there exists a polynomial $Q_x$ of degree at most $n$ such that
\[\,r^{-d/p}\,\|F-Q_x\|_{L^p(B(x,r))}= o(r^{n})\quad \text{when } r\rightarrow 0^+.\]
Moreover, for almost every $x\in E_{1/m}$,  
\[r^{-d/p}\,\|f-F\|_{L^p(B(x,r))}\leq C\,r^n.\]
Using Lemma \ref{lemme 2 Rademacher}, we get
\[r^{-d/p}\|f-F\|_{L^p(B(x,r))}= o(r^n)\text{ when } r\rightarrow 0^+,\]
Thus, for almost all $x\in E_{1/m}$, 
\begin{align*}
r^{-d/p}\,\|f-Q_x\|_{L^p(B(x,r))}&\leq r^{-d/p}\,\|f-F\|_{L^p(B(x,r))}+r^{-d/p}\,\|F-Q_x\|_{L^p(B(x,r))}\\
&= o(\phi(r)) + o(r^{n+1})\\&= o(\phi(r))\text{ when } r\rightarrow 0^+.
\end{align*}
We have thus proved the result for sets $E_{1/m}$, $m\in\mathbb{N}$. Since a countable union of negligible sets is negligible, we deduce the conclusion for $E$.
\smallskip\\$\bullet$ If $E$ is unbounded, it suffices to apply the previous case to the sets
\[
E^{(j)} = E \cap B(0,j), \quad j \in \mathbb{N},
\]
and to conclude again using the fact that a countable union of negligible sets is negligible.
\end{proof}

\section{Rademacher's Theorem for Calderon-Zygmund-type Spaces}\label{sec4}
We suppose that $\phi\in\mathcal{B}$ and $n\in\mathbb{N}_0$ are such that $n<\underline{b}(\phi)\leq \overline{b}(\phi)<n+1$. For those fractional indices, we can't easily adapt the proof because of Theorem \ref{Thm 2 ziemer sobolev}. The idea to get results is to reduce to the case of integer indices. Before the Theorem, let us remark that if we have some assumption on $f$ at a neighborhood of $x_0$, then $f\in t^p_{\phi,n+1}(x_0)$.
\begin{remark}
If $f \in C^{n+1}(V)$, where $V$ is an open neighborhood of $x_0$, then $$f \in t^p_{\phi,n}(x_0)\subseteq t^p_{\phi,n+1}(x_0).$$ Indeed, there exists $R > 0$ such that $\overline{B(x_0,R)} \subseteq V$. If $P$ denotes the Taylor polynomial of $f$ of order $n$ at $x_0$, then since $f$ is of class $C^{n+1}$ on $V$, there exists a constant $C > 0$ such that
\[
|f(x)-P(x)| \leq C|x-x_0|^{n+1} \quad \forall x \in B(x_0,R).
\]
Thus, for every $r \in (0,R]$, we have
\[
r^{-d/p}\|f-P\|_{L^p(B(x_0,r))} \leq C'' r^{n+1}.
\]
By assumption, there exists $\varepsilon > 0$ such that $\overline{b}(\phi) + \varepsilon < n+1$.  
It follows that
\[
\phi(r)^{-1} r^{-d/p}\|f\|_{L^p(B(x_0,r))}
\leq \phi(r)^{-1} C'' r^{n+1}
\leq \frac{C''}{C_1} r^{\,n+1-(\overline{b}(\phi)+\varepsilon)} \underset{r \to 0^+}{\longrightarrow} 0,
\]
and hence $f \in t^p_{\phi,n}(x_0)$.
\end{remark}

Note that, 
\[t^p_{\overline{b}(\phi)}(x_0)\subseteq t^p_{\phi,n}(x_0)\subseteq t^p_{\phi,n+1}(x_0)\subseteq t^p_{\underline{b}(\phi)}(x_0).\]
Since the second inclusion is often strict, the best general Rademacher's Theorem we can get is the following.
\begin{theorem}\label{Rademacher's Theorem for Calderon-Zygmund-type Spaces}\emph{(Rademacher's Theorem for Calderon-Zygmund-type Spaces)}
\\Let $E\subseteq \mathbb{R}^d$ measurable, $p\in(1,\infty)$, $\phi\in\mathcal{B}$ and $n\in\mathbb{N}_0$ such that $n<\underline{b}(\phi)\leq \overline{b}(\phi)<n+1$. If $f\in T^p_\phi(x)$ for all $x\in E$, then $f\in t^p_{\phi,n+1}(x)$ for almost all $x\in E$.
\end{theorem}
\begin{proof}~
\smallskip\\$\bullet$ Let us first assume that $E$ is bounded. By lemma \ref{lemme 1 Rademacher} and Lusin's Theorem, there exist compacts $E_{1/m}\subseteq E$, $m\in\mathbb{N}$, such that
\[\mathcal{L}\left(E\backslash E_{1/m}\right)\leq\frac{1}{m}\]
and such that \[\xi_f : E_{1/m} \rightarrow [0,\infty) : x\mapsto \|f\|_{T^p_\phi(x)}\] is continuous. Let $m\in\mathbb{N}$, there exists a constant $M=M(m)>0$ such that
 \[\|f\|_{T^p_\phi(x)}\leq M\quad\forall x\in E_{1/m}.\]
Let $x\in E_{1/m}$, we denote by $P_x$ the polynomial of degree strictly less than or equal to $n$ associated with $f$. Thanks to Whitney's extension Theorem \ref{Gen Whitney}, there exists an open set $U$ containing $E_{1/m}$ and a function $F\in C^n(U)$ such that $F=f$ almost everywhere $E_{1/m}$ and such that for all $|\beta|\leq n$ and for almost all $x\in E_{1/m}$,
\[D^\beta F(x)=D^\beta P_x(x).\]
By Theorem \ref{Thm 1 ziemer sobolev}, $F\in W^p_{n+1,\text{loc}}(\mathbb{R}^d)$. Using Theorem \ref{Thm 2 ziemer sobolev} with $m_1=m_2=n+1$, it follows that for almost every $x\in \mathbb{R}^d$, there exists a polynomial $Q_x$ of degree at most $n$ such that
\[\,r^{-d/p}\,\|F-Q_x\|_{L^p(B(x,r))}= o(r^{n+1})\quad \text{when } r\rightarrow 0^+.\]
Moreover, for almost every $x\in E_{1/m}$,  
\[r^{-d/p}\,\|f-F\|_{L^p(B(x,r))}\leq C\,\phi(r).\]
Using Lemma \ref{lemme 2 Rademacher}, we get
\[r^{-d/p}\|f-F\|_{L^p(B(x,r))}= o(\phi(r))\text{ when } r\rightarrow 0^+,\]
Thus, for almost all $x\in E_{1/m}$, 
\begin{align*}
r^{-d/p}\,\|f-Q_x\|_{L^p(B(x,r))}&\leq r^{-d/p}\,\|f-F\|_{L^p(B(x,r))}+r^{-d/p}\,\|F-Q_x\|_{L^p(B(x,r))}\\
&= o(\phi(r)) + o(r^{n+1})\\&= o(\phi(r))\text{ when } r\rightarrow 0^+.
\end{align*}
We have thus proved the result for sets $E_{1/m}$, $m\in\mathbb{N}$. Since a countable union of negligible sets is negligible, we deduce the conclusion for $E$.
\smallskip\\$\bullet$ If $E$ is unbounded, it suffices to apply the previous case to the sets
\[
E^{(j)} = E \cap B(0,j), \quad j \in \mathbb{N},
\]
and to conclude again using the fact that a countable union of negligible sets is negligible.
\end{proof}

\begin{example}\label{Example 1}
Suppose that $\overline{b}(\phi)<h_p^f(x)$ for all $x\in E$, then the statement
\[f\in T^p_\phi(x) \quad\forall x\in E\implies f\in t^p_{\phi,n}(x) \quad\text{for almost all }x\in E\]
is always true since, in this case, $f\in t^p_{\phi,n}(x)$ for all $x\in E$.
\end{example}

\begin{example}\label{Example 2}
The following observation with Brownian motion led us to conclude that Rademacher's theorem does not hold for fractional indices in the space $t^p_{\phi,0}$ (as opposed to $t^p_{\phi,1}$). Indeed, using (\ref{equa Brownien}), it is clear that almost surely and for almost every $t_0 \in \mathbb{R}$, $B\in T^p_\phi(t_0)$, where $\phi(t)\sim (t \log \log 1/|t|)^{1/2}$ for all $t>0$ but 
almost surely and for almost every $t_0 \in \mathbb{R}$, 
\[\limsup_{r\rightarrow 0^+}\frac{r^{-1/p}(\int_{B(t_0,r)}|B(t)-B(t_0)|^pdt)^{1/p}}{\sqrt{r\log \log 1/r}}=\sqrt{2}. \]
\end{example}

\section{Conclusion}
Let $\phi\in\mathcal{B}$ and $n\in\mathbb{N}_0$ are such that $n<\underline{b}(\phi)\leq \overline{b}(\phi)<n+1$. On the one hand, we have the inclusions
\[
t^p_{\phi,n}(x_0) \subseteq T^p_{\phi}(x_0)
\quad \text{and} \quad
t^p_{\phi,n}(x_0) \subseteq t^p_{\phi,n+1}(x_0),
\]
but in general we do not have the inclusion
\[
t^p_{\phi,n+1}(x_0) \subseteq T^p_{\phi}(x_0).
\]

On the other hand, our Rademacher's Theorem \ref{Rademacher's Theorem for Calderon-Zygmund-type Spaces} states that
\[
f \in T^p_\phi(x) \,\,\,\,\,\, \forall x \in E
\quad \implies \quad
f \in t^p_{\phi,n+1}(x) \quad \text{for almost all} \ x \in E.
\]

Moreover, Example \ref{Example 1} shows that this latter result is only of pointwise interest when 
$\overline{b}(\phi) \geq h_p^f(x)$, since otherwise one already has 
$f \in t^p_{\phi,n}(x)$. Example \ref{Example 2} further illustrates the limitations of our result, and shows that in general one cannot expect $f$ to belong almost everywhere to $t^p_{\phi,n}(x)$.

\section*{Declarations}

\begin{itemize}
\item Funding : The author did not receive any specific funding for this work.
\item Ethics approval : Not applicable.
\item Availability of data and materials : Not applicable.
\end{itemize}


\begin{thebibliography}{99}

\bibitem{Boyd:67}
D.~W. Boyd.
\newblock The {H}ilbert {T}ransform on {R}earrangement-{I}nvariant {S}paces.
\newblock {\em Canadian J. Math.}, 19, 1967.

\bibitem{Boyd:69}
D.~W. Boyd.
\newblock Indices of function spaces and their relationship to interpolation.
\newblock {\em Canadian J. Math.}, 21:1245--1254, 1969.

\bibitem{Calderon:61}
A.~P. Calder{\'o}n and A.~Zygmund.
\newblock Local properties of solutions of elliptic partial differential equations.
\newblock {\em Studia Math.}, 20:181--225, 1961.

\bibitem{Cheeger:99}
J.~Cheeger.
\newblock Differentiability of {L}ipschitz functions on metric measure spaces.
\newblock {\em Geom. Funct. Anal.}, 9(3):428--517, 1999.

\bibitem{Evans:15}
L.~C. Evans and R.~F. Gariepy.
\newblock {\em Measure theory and fine properties of functions}.
\newblock Textbooks in Mathematics, revised edition of the 1992 original ed., CRC Press, Boca Raton, FL, 2015.

\bibitem{Jaffard:18}
S.~Jaffard and B.~Martin.
\newblock Multifractal analysis of the {B}rjuno function.
\newblock {\em Invent. Math.}, 212:109--132, 2018.

\bibitem{Johnson:79}
W.~B. Johnson, B.~Maurey, G.~Schechtman, and L.~Tzafriri.
\newblock {\em Symmetric structures in {B}anach spaces}.
\newblock Memoirs of the Amer. Math. Soc., Providence, Rhode Island, 1979.

\bibitem{Krein:82}
S.~G. Krein, Ju.~I. Petunin, and E.~M. Semenov.
\newblock {\em Interpolation of {L}inear {O}perators}, volume~54.
\newblock Translations of Mathematical Monographs, Amer. Math. Soc., Providence, Rhode Island, 1982.

\bibitem{Lamby:21}
T.~Lamby.
\newblock Espaces de {C}alder{\'o}n-{Z}ygmund et r{\'e}gularit{\'e} de fonctions non localement born{\'e}es.
\newblock Master's thesis, Universit{\'e} de Li{\`e}ge, 2021.

\bibitem{Lamby:22}
T.~Lamby and S.~Nicolay.
\newblock Some remarks on the {B}oyd functions related to the admissible sequences.
\newblock {\em Z. Anal. Anwend.}, 41:211--227, 2022.

\bibitem{Lamby:24}
T.~Lamby and S.~Nicolay.
\newblock Interpolation with a function parameter from the category point of view.
\newblock {\em J. Funct. Anal.}, 286, 2024.

\bibitem{Loosveldt:20}
L.~Loosveldt and S.~Nicolay.
\newblock Generalized {$T_u^p$} spaces: On the trail of {C}alder{\'o}n and {Z}ygmund.
\newblock {\em Diss. Math.}, 554, 2020.

\bibitem{Luzin:12}
N.~Luzin.
\newblock Sur les propri{\'e}t{\'e}s des fonctions mesurables.
\newblock {\em Comptes rendus hebdomadaires des s{\'e}ances de l'Acad{\'e}mie des sciences}, 154:1688--1690, 1912.

\bibitem{Merucci:84}
C.~Merucci.
\newblock Applications of interpolation with a function parameter to {L}orentz, {S}obolev and {B}esov spaces.
\newblock In M.~Cwikel and J.~Peetre, editors, {\em Interpolation Spaces and Allied Topics in Analysis}, pages 183--201. Springer, Proceedings of the Conference Held in Lund, Sweden, August 29--September 1, 1983, 1984.

\bibitem{Morters:10}
P.~M{\"o}rters and Y.~Peres.
\newblock {\em Brownian Motion}.
\newblock Cambridge Series in Statistical and Probabilistic Mathematics, vol.~30, Cambridge University Press, Cambridge, 2010.

\bibitem{Seuret:17}
S.~Seuret and A.~Ubis.
\newblock Local {$L^2$}-regularity of {R}iemann's {F}ourier series.
\newblock {\em Ann. Inst. Fourier}, 67:2237--2264, 2017.

\bibitem{Ziemer:89}
W.~P. Ziemer.
\newblock {\em Weakly differentiable functions}.
\newblock Graduate Texts in Mathematics, vol.~120, Springer, New York, 1989.


\end{thebibliography}
\end{document}